\documentclass{amsart}
\usepackage{amssymb,amsmath,amsthm}
\usepackage[dvips]{graphicx}

\let\oldmarginpar\marginpar
\renewcommand\marginpar[1]{\-\oldmarginpar[\raggedleft\footnotesize #1]%
{\raggedright\footnotesize #1}}

\begin{document}

\newtheorem{theorem}{Theorem}[section]
\newtheorem{corollary}[theorem]{Corollary}
\newtheorem{lemma}[theorem]{Lemma}
\newtheorem{proposition}[theorem]{Proposition}
\theoremstyle{definition}
\newtheorem{definition}[theorem]{Definition}
\theoremstyle{remark}
\newtheorem{remark}[theorem]{Remark}
\theoremstyle{definition}
\newtheorem{example}[theorem]{Example}

\numberwithin{equation}{section}

\def\R{{\mathbb R}}
\def\H{{\mathbb H}}
\def\rank{{\text{rank}\,}}
\def\bd{{\partial}}

\title[Pointwise almost h-semi-slant submanifolds]{Pointwise almost h-semi-slant submanifolds}

\author{Kwang-Soon Park}
\address{Sogang research team for discrete and geometric structures, Sogang University, Seoul 121-742, Republic of Korea}
\email{parkksn@gmail.com}
\thanks{This research was supported by Sogang Research Team for Discrete and Geometric Structures.}



\begin{abstract}
We introduce the notions of pointwise almost h-slant submanifolds and pointwise almost h-semi-slant submanifolds as a
generalization of slant submanifolds, pointwise slant submanifolds, semi-slant submanifolds, and pointwise semi-slant submanifolds.
We obtain a characterization and investigate the following: the integrability of distributions,
the conditions for such distributions to be totally geodesic foliations,
the mean curvature vector fields on totally umbilic submanifolds, the properties of h-slant functions and h-semi-slant functions, the properties
of non-trivial warped product proper pointwise h-semi-slant submanifolds. We also obtain the topological properties
of proper pointwise almost h-slant submanifolds and give an inequality for the squared norm of the second fundamental form in terms of
a warping function and a h-semi-slant function for a warped product submanifold of a hyperk\"{a}hler manifold.
Finally, we give some examples of such submanifolds.
\end{abstract}

\maketitle
\section{Introduction}\label{intro}
\addcontentsline{toc}{section}{Introduction}

Given a Riemannian manifold $(\overline{M},g)$ with some additional structures, there are several kinds of submanifolds:

(almost) complex submanifolds (\cite{K}, \cite{G11}, \cite{K0}, \cite{N}),
totally real submanifolds (\cite{CO}, \cite{E0}, \cite{CDVV}, \cite{CN}),
CR submanifolds (\cite{YK}, \cite{B0}, \cite{C333}, \cite{S0}),
QR submanifolds (\cite{B00}, \cite{KP}, \cite{BF}, \cite{GSK}),
slant submanifolds (\cite{C}, \cite{L}, \cite{CCFF}, \cite{MMT}, \cite{S}),
pointwise slant submanifolds \cite{CG},
semi-slant submanifolds (\cite{P}, \cite{S00}, \cite{CCFF0}, \cite{KKK}),
pointwise semi-slant submanifolds \cite{S2}, etc.

In 1990, B. Y. Chen \cite{C} defined the notion of slant submanifolds of an almost Hermitian manifold as a generalization of
almost complex submanifolds and totally real submanifolds.

In 1994, N. Papaghiuc \cite{P} introduced a semi-slant submanifold of an almost Hermitian manifold as a generalization of
CR-submanifolds and slant submanifolds.

In 1998, F. Etayo \cite{E1} defined the notion of pointwise slant submanifolds of an almost Hermitian manifold
under the name of quasi-slant submanifolds as a generalization of slant submanifolds.

In 2012, B. Y. Chen and O. J. Garay \cite{CG} studied deeply pointwise slant submanifolds of an almost Hermitian manifold.

In 2013, B. Sahin \cite{S2} introduced pointwise semi-slant submanifolds of an almost Hermitian manifold
(for further information, see (\cite{CG}, \cite{S2})).

As we know, the quaternionic K\"{a}hler manifolds have applications in physics as the target
spaces for nonlinear $\sigma-$models with supersymmetry \cite{CMMS}.

As a generalization of slant submanifolds, pointwise slant submanifolds, semi-slant submanifolds, and pointwise semi-slant submanifolds,
we will define the notions of pointwise almost h-slant submanifolds and pointwise almost h-semi-slant submanifolds.

The paper is organized as follows. In section 2 we recall some notions, which are needed in the following sections.
In section 3 we give the definitions of pointwise almost h-slant submanifolds, pointwise almost h-semi-invariant submanifolds,
and pointwise almost h-semi-slant submanifolds and obtain some properties:
a characterization of such submanifolds, the integrability of distributions, the equivalent conditions for such distributions to be
totally geodesic foliations, the mean curvature vector fields on totally umbilic submanifolds, etc. In section 4 we deal with the geometrical
properties of slant functions. In section 5 we consider proper pointwise almost h-slant submanifolds of hyperk\"{a}hler manifolds
and obtain some subalgebras of their cohomology algebras. From these topological results, we induce some geometrical properties.
In section 6 we get some properties on non-trivial warped product proper pointwise h-semi-slant submanifolds of hyperk\"{a}hler manifolds.
Using these results, we obtain an inequality for the squared norm of the second fundamental form in terms of
a warping function and a h-semi-slant function.
In section 7 we give some examples of such submanifolds.

\section{Preliminaries}\label{Prel}

Let $(\overline{M}, g)$ be a Riemannian manifold, where $\overline{M}$ is a $m$-dimensional $C^{\infty}$-manifold
and $g$ is a Riemannian metric on $\overline{M}$.
Let $M$ be a $n$-dimensional submanifold of $(\overline{M}, g)$.

Denote by $TM^{\perp}$ the normal bundle of $M$ in $\overline{M}$.

Denote by $\nabla$ and $\overline{\nabla}$ the Levi-Civita connections of $M$ and $\overline{M}$, respectively.

Then the {\em Gauss} and {\em Weingarten formulas} are given by
\begin{eqnarray}
  \overline{\nabla}_X Y & = & \nabla_X Y + h(X, Y), \label{eq: gauss} \\
  \overline{\nabla}_X Z & = & -A_Z X + D_X Z, \label{eq: weing}
\end{eqnarray}
respectively, for tangent vector fields $X,Y\in \Gamma(TM)$ and a normal vector field $Z\in \Gamma(TM^{\perp})$, where $h$ denotes the {\em second fundamental
form}, $D$ the {\em normal connection}, and $A$ the {\em shape operator} of $M$ in $\overline{M}$.

Then the second fundamental form and the shape operator are related by
\begin{equation}\label{eq: shape}
\langle A_Z X, Y \rangle = \langle h(X, Y), Z \rangle,
\end{equation}
where $\langle \ , \ \rangle$ denotes the induced metric on $M$ as well as the Riemannian metric $g$ on $\overline{M}$.
We will also denote by $g$ the induced metric on $M$.

Choose a local orthonormal frame
$\{ e_1, \cdots, e_m \}$ of $T\overline{M}$ such that $e_1, \cdots, e_n$ are tangent to $M$ and $e_{n+1}, \cdots, e_m$ are normal to $M$.

Then the {\em mean curvature vector} $H$ is defined by
\begin{equation}\label{eq: mean}
H := \frac{1}{n} trace \ h = \frac{1}{n}\sum_{i=1}^n h(e_i, e_i)
\end{equation}
and the {\em squared mean curvature} is given by $H^2 := \langle H, H \rangle$.

The {\em squared norm of the second
fundamental form} $h$ is given by
\begin{equation}\label{eq: sqnorm}
|| h ||^2 := \sum_{i,j=1}^n \langle h(e_i, e_j), h(e_i, e_j) \rangle.
\end{equation}
Let $(B, g_B)$ and $(F, g_F)$ be Riemannian manifolds, where $g_B$ and $g_F$ are Riemannian metrics on manifolds $B$ and $F$, respectively.
Let $f$ be a positive $C^{\infty}$-function on $B$. Consider the product manifold $B\times F$ with the natural projections $\pi_1 : B\times F \mapsto B$
 and $\pi_2 : B\times F \mapsto F$. The {\em warped product manifold} $M=B\times_f F$ is the product manifold $B\times F$ equipped with the Riemannian
 metric $g_M$ such that
\begin{equation}\label{eq: wmetr}
|| X ||^2 = || d\pi_1 X ||^2 + f^2(\pi_1(x)) || d\pi_2 X ||^2
\end{equation}
for any tangent vector $X\in T_x M$, $x\in M$. Thus, we get $g_M = g_B + f^2g_F$. The function $f$ is called the {\em warping function} of the warped product
manifold $M$ \cite {C6}.

If the warping function $f$ is constant, then we call the warped product manifold $M$ {\em trivial}.

Given vector fields $X\in \Gamma(TB)$ and $Y\in \Gamma(TF)$, we can obtain their natural horizontal lifts $\widetilde{X},\widetilde{Y}\in \Gamma(TM)$
such that $d\pi_1 \widetilde{X} = X$ and $d\pi_2 \widetilde{Y} = Y$.

Conveniently, identify $\widetilde{X}$ and $\widetilde{Y}$
with $X$ and $Y$, respectively.

We choose a local orthonormal frame $\{ e_1, \cdots, e_n \}$ of
the tangent bundle $TM$ of $M$ such that $e_1, \cdots, e_{n_1}\in \Gamma(TB)$ and $e_{n_1+1}, \cdots, e_{n}\in \Gamma(TF)$.

Then we have
\begin{equation}\label{eq: lapl2}
\triangle f = \sum_{i=1}^{n_1} ((\nabla_{e_i} e_i)f - e_i^2 f).
\end{equation}
Given unit vector fields $X,Y\in \Gamma(TM)$ such that $X\in \Gamma(TB)$ and $Y\in \Gamma(TF)$, we easily obtain
\begin{equation}\label{eq: warp2}
\nabla_X Y = \nabla_Y X = (X \ln f)Y,
\end{equation}
where $\nabla$ is the Levi-Civita connection of $(M, g_M)$, so that
\begin{eqnarray}
K(X\wedge Y)
& = & \langle \nabla_Y\nabla_X X - \nabla_X\nabla_Y X, Y \rangle   \label{sect3}    \\
& = & \frac{1}{f} ((\nabla_X X)f - X^2 f),  \nonumber
\end{eqnarray}
where $K(X\wedge Y)$ denotes the sectional curvature of the plane $<X, Y>$ spanned by $X$ and $Y$ over $\mathbb{R}$.

Hence,
\begin{equation}\label{eq: sect4}
\frac{\triangle f}{f} = \sum_{i=1}^{n_1} K(e_i\wedge e_j)
\end{equation}
for each $j = n_1+1, \cdots, n$.

We remind some notions in the submanifold theory. It can help us to understand our notions.
We may compare these notions with our notions and find some similarities and differences among them.

Let $(\overline{M}, g, J)$ be an almost Hermitian manifold, where $\overline{M}$ is a $C^{\infty}$-manifold,
$g$ is a Riemannian metric on $\overline{M}$, and $J$ is an almost complex structure on $\overline{M}$ which is compatible with $g$.
i.e., $J\in \text{End}(T \overline{M})$, $J^2 = -id$, $g(JX, JY) = g(X, Y)$ for $X,Y\in \Gamma(T \overline{M})$.

Let $M$ be a submanifold of $\overline{M} = (\overline{M}, g, J)$.

We call $M$ an {\em almost complex submanifold} of $\overline{M}$ if $J(T_x M)\subset T_x M$ for $x\in M$.

The submanifold $M$ is said to be a {\em totally real submanifold} if $J(T_x M)\subset T_x M^{\perp}$ for $x\in M$.

The submanifold $M$ is called a {\em CR-submanifold} \cite{B0} if there exists a distribution $\mathcal{D} \subset TM$ on $M$
such that $J(\mathcal{D}_x) = \mathcal{D}_x$ and $J(\mathcal{D}_x^{\perp}) \subset T_x M^{\perp}$ for $x\in M$,
where $\mathcal{D}^{\perp}$ is the orthogonal complement of $\mathcal{D}$ in $TM$.

We call $M$ a {\em slant submanifold} \cite{C} of $\overline{M}$ if the angle $\theta = \theta(X)$ between $JX$ and the tangent space
$T_x M$ is constant for nonzero $X\in T_x M$ and any $x\in M$.

The submanifold $M$ is said to be a {\em semi-slant submanifold} \cite{P} if there is a distribution $\mathcal{D} \subset TM$ on $M$
such that $J(\mathcal{D}_x) = \mathcal{D}_x$ for $x\in M$ and the angle $\theta = \theta(X)$ between $JX$ and the space
$\mathcal{D}_x^{\perp}$ is constant for nonzero $X\in \mathcal{D}_x^{\perp}$ and any $x\in M$,
where $\mathcal{D}^{\perp}$ is the orthogonal complement of $\mathcal{D}$ in $TM$.

The submanifold $M$ is called a {\em pointwise slant submanifold} (\cite{E1}, \cite{CG}) if at each given point $x\in M$,
the angle $\theta = \theta(X)$ between $JX$ and the tangent space
$T_x M$ is constant for nonzero $X\in T_x M$.

We call $M$ a {\em pointwise semi-slant submanifold} \cite{S2} of $\overline{M}$ if there is a distribution $\mathcal{D} \subset TM$ on $M$
such that $J(\mathcal{D}_x) = \mathcal{D}_x$ for $x\in M$ and at each given point $x\in M$, the angle $\theta = \theta(X)$ between $JX$ and the space
$\mathcal{D}_x^{\perp}$ is constant for nonzero $X\in \mathcal{D}_x^{\perp}$,
where $\mathcal{D}^{\perp}$ is the orthogonal complement of $\mathcal{D}$ in $TM$.

Now, we introduce almost quaternionic Hermitian manifolds.

Let $\overline{M}$ be a $4m-$dimensional $C^{\infty}$-manifold and let $E$ be a rank 3 subbundle of
$\text{End} (T\overline{M})$ such that for any point $p\in \overline{M}$ with a
neighborhood $U$, there exists a local basis $\{ J_1,J_2,J_3 \}$
of sections of $E$ on $U$ satisfying for all $\alpha\in \{ 1,2,3 \}$
$$
J_{\alpha}^2=-id, \quad
J_{\alpha}J_{\alpha+1}=-J_{\alpha+1}J_{\alpha}=J_{\alpha+2},
$$
where the indices are taken from $\{ 1,2,3 \}$ modulo 3.

Then we call $E$ an {\em almost quaternionic structure} on $\overline{M}$ and $(\overline{M},E)$ an {\em almost quaternionic manifold} \cite{AM}.

Moreover, let $g$ be a Riemannian metric on $\overline{M}$ such that for any point $p\in \overline{M}$ with a
neighborhood $U$, there exists a local basis $\{ J_1,J_2,J_3 \}$
of sections of $E$ on $U$ satisfying for all $\alpha\in \{ 1,2,3 \}$
\begin{equation}\label{hypercom}
J_{\alpha}^2=-id, \quad
J_{\alpha}J_{\alpha+1}=-J_{\alpha+1}J_{\alpha}=J_{\alpha+2},
\end{equation}
\begin{equation}\label{hypermet}
g(J_{\alpha}X, J_{\alpha}Y)=g(X, Y)
\end{equation}
for all vector fields  $X, Y\in \Gamma(T\overline{M})$, where the indices are
taken from $\{ 1,2,3 \}$ modulo 3.

Then we call $(\overline{M},E,g)$ an {\em almost quaternionic Hermitian manifold} \cite{IMV}.

Conveniently, the above basis $\{ J_1,J_2,J_3 \}$ satisfying (\ref{hypercom}) and (\ref{hypermet}) is said to be a {\em quaternionic Hermitian basis}.

Let $(\overline{M},E,g)$ be an almost quaternionic Hermitian manifold.

We call $(\overline{M},E,g)$ a {\em quaternionic K\"{a}hler manifold} if there exist
locally defined 1-forms $\omega_1, \omega_2, \omega_3$ such that
for $\alpha \in \{ 1,2,3 \}$
$$
\nabla_X J_{\alpha} =
\omega_{\alpha+2}(X)J_{\alpha+1}-\omega_{\alpha+1}(X)J_{\alpha+2}
$$
for any vector field $X\in \Gamma(T\overline{M})$, where the indices are
taken from $\{ 1,2,3 \}$ modulo 3 \cite{IMV}.

If there exists a global parallel quaternionic Hermitian basis $\{ J_1,J_2,J_3 \}$ of
sections of $E$ on $\overline{M}$ (i.e., $\nabla J_{\alpha} = 0$ for $\alpha \in \{ 1,2,3 \}$, where $\nabla$ is the Levi-Civita connection of the
metric $g$), then $(\overline{M}, E, g )$ is said to be a {\em hyperk\"{a}hler manifold}.
Furthermore, we call $(J_1, J_2, J_3, g )$ a {\em hyperk\"{a}hler structure} on $\overline{M}$ and $g$ a {\em hyperk\"{a}hler metric} \cite{B}.

Let $\overline{M} = (\overline{M},E,g)$ be an almost quaternionic Hermitian manifold and $M$ a submanifold of $\overline{M}$.

We call $M$ a {\em QR-submanifold} (quaternionic-real submanifold) \cite{B00} of $\overline{M}$ if there exists a vector subbundle $\mathcal{D}$
of $TM^{\perp}$ on $M$ such that given a local quaternionic Hermitian basis $\{ J_1,J_2,J_3 \}$ of $E$, we have $J_{\alpha} \mathcal{D} = \mathcal{D}$
and $J_{\alpha} (\mathcal{D}^{\perp}) \subset TM$ for $\alpha \in \{ 1,2,3 \}$,
where $\mathcal{D}^{\perp}$ is the orthogonal complement of $\mathcal{D}$ in $TM^{\perp}$.

The submanifold $M$ is said to be a {\em quaternion CR-submanifold} \cite{BCU} if there exists a distribution $\mathcal{D} \subset TM$
 on $M$ such that given a local quaternionic Hermitian basis $\{ J_1,J_2,J_3 \}$ of $E$, we get $J_{\alpha} \mathcal{D} = \mathcal{D}$
and $J_{\alpha} (\mathcal{D}^{\perp}) \subset TM^{\perp}$ for $\alpha \in \{ 1,2,3 \}$,
where $\mathcal{D}^{\perp}$ is the orthogonal complement of $\mathcal{D}$ in $TM$ (for more information, see \cite{BF}).

Throughout this paper, we will use the above notations.

\section{Pointwise almost h-semi-slant submanifolds}\label{semi}

In this section we define the notions of pointwise h-slant submanifolds, pointwise almost h-slant submanifolds,
pointwise h-semi-invariant submanifolds, pointwise almost h-semi-invariant submanifolds,
pointwise h-semi-slant submanifolds, and pointwise almost h-semi-slant submanifolds.
And we investigate the geometry of such submanifolds.

\begin{definition}
Let $(\overline{M},E,g)$ be an almost quaternionic Hermitian manifold and
$M$ a submanifold of $(\overline{M},g)$. The submanifold $M$ is called a {\em pointwise almost h-slant submanifold}
if given a point $p\in M$ with a neighborhood
$V$, there exist an open set $U\subset \overline{M}$ with $U\cap M = V$ and a quaternionic Hermitian basis $\{ I,J,K \}$ of
sections of $E$ on $U$ such that for each $R\in \{ I,J,K \}$, at each given point $q\in V$ the angle $\theta_R=\theta_R(X)$
between $RX$ and the tangent space $T_q M$ is constant for nonzero $X\in T_q M$.
\end{definition}

We call such a basis $\{ I,J,K \}$ a {\em pointwise almost h-slant basis} and
the angles $\{ \theta_I,\theta_J,\theta_K \}$ {\em almost h-slant functions} as functions on $V$.

\begin{remark}
Furthermore, if we have
$$
\theta=\theta_I=\theta_J=\theta_K,
$$
then we call the submanifold $M$ a {\em pointwise h-slant submanifold}, the basis $\{ I,J,K \}$ a {\em pointwise h-slant basis},
and the angle $\theta$ a {\em h-slant function}.
\end{remark}

\begin{definition}
Let $(\overline{M},E,g)$ be an almost quaternionic Hermitian manifold and
$M$ a submanifold of $(\overline{M},g)$. The submanifold $M$ is called a {\em pointwise almost h-semi-slant submanifold} if given a point $p\in M$ with a neighborhood $V$, there exist an open set $U\subset \overline{M}$ with $U\cap M = V$ and a  quaternionic Hermitian basis $\{ I,J,K \}$ of
sections of $E$ on $U$ such that for each $R\in \{ I,J,K \}$,
there is a distribution $\mathcal{D}_1^R \subset TM$ on $V$
such that
$$
TM =\mathcal{D}_1^R\oplus \mathcal{D}_2^R, \
R(\mathcal{D}_1^R)=\mathcal{D}_1^R,
$$
and at each given point $q\in V$ the angle $\theta_R=\theta_R(X)$ between $RX$ and the space
$(\mathcal{D}_2^R)_q$ is constant for nonzero $X\in
(\mathcal{D}_2^R)_q$, where $\mathcal{D}_2^R$ is the
orthogonal complement of $\mathcal{D}_1^R$ in $TM$.
\end{definition}

We call such a basis $\{ I,J,K \}$ a {\em pointwise almost h-semi-slant
basis} and the angles $\{ \theta_I,\theta_J,\theta_K \}$ {\em
almost h-semi-slant functions} as functions on $V$.

\begin{remark}
Furthermore, if $\mathcal{D}_1 = \mathcal{D}_1^I = \mathcal{D}_1^J = \mathcal{D}_1^K$,
then we call the submanifold $M$ a {\em pointwise h-semi-slant submanifold}, the basis $\{ I,J,K \}$ a {\em pointwise h-semi-slant basis},
and the angles $\{ \theta_I,\theta_J,\theta_K \}$ {\em h-semi-slant functions}.

Moreover, if we have
$$
\theta=\theta_I=\theta_J=\theta_K,
$$
then we call the submanifold $M$ a {\em pointwise strictly h-semi-slant submanifold}, the basis $\{ I,J,K \}$ a {\em pointwise strictly h-semi-slant basis},
and the angle $\theta$ a {\em strictly h-semi-slant function}.
\end{remark}

\begin{remark}
Let $M$ be a submanifold of an almost quaternionic Hermitian manifold $(\overline{M},E,g)$.
If $M$ is a pointwise almost h-semi-slant submanifold of $(\overline{M},E,g)$ with almost h-semi-slant functions
$\theta_I = \theta_J = \theta_K = \frac{\pi}{2}$,
then we call the submanifold $M$ a {\em pointwise almost h-semi-invariant submanifold} and the basis $\{ I,J,K \}$
a {\em pointwise almost h-semi-invariant basis}.

Similarly, if $M$ is a pointwise h-semi-slant submanifold of $(\overline{M},E,g)$ with h-semi-slant functions
$\theta_I = \theta_J = \theta_K = \frac{\pi}{2}$,
then we call the submanifold $M$ a {\em pointwise h-semi-invariant submanifold} and the basis $\{ I,J,K \}$
a {\em pointwise h-semi-invariant basis}.
\end{remark}

\begin{remark}
Let $M$ be a submanifold of an almost quaternionic Hermitian manifold $(\overline{M},E,g)$.
If $M$ is a pointwise h-semi-slant submanifold of $(\overline{M},E,g)$ with $\mathcal{D}_1 = 0$ and $\mathcal{D}_2 = TM$,
then the submanifold $M$ is a pointwise almost h-slant submanifold of $(\overline{M},E,g)$.
\end{remark}

\begin{remark}
As we know, an isometric immersion is the corresponding notion of a Riemannian submersion (\cite{O}, \cite{G}, \cite{FIP}).
In the similar way, almost h-slant submanifolds, almost h-semi-invariant submanifolds, and almost h-semi-slant submanifolds correspond to
almost h-slant submersions, almost h-semi-invariant submersions, and almost h-semi-slant submersions, respectively (\cite{P2}, \cite{P22}, \cite{P3}).
Similarly, pointwise almost h-slant submanifolds, pointwise almost h-semi-invariant submanifolds, and pointwise almost h-semi-slant submanifolds are
the corresponding notions of pointwise almost h-slant submersions, pointwise almost h-semi-invariant submersions,
and pointwise almost h-semi-slant submersions, respectively.
Therefore, we need to study deeply the notions of almost h-slant submanifolds, almost h-semi-invariant submanifolds,  almost h-semi-slant submanifolds,
pointwise almost h-slant submersions, pointwise almost h-semi-invariant submersions,  and pointwise almost h-semi-slant submersions as the future works.
\end{remark}

On a K\"{a}hler manifold $(\overline{M},g,J)$, we can consider a pointwise slant submanifold as a generalization of a slant submanifold
and a pointwise semi-slant submanifold as a generalization of a semi-slant submanifold. As we know, if we can get some functions on $\overline{M}$,
then by using these functions we may obtain some information about the manifold $\overline{M}$ and they are surely helpful to us to study
the manifold $\overline{M}$. In this point of view, pointwise slant submanifolds and pointwise semi-slant submanifolds may have much information than
slant submanifolds and semi-slant submanifolds, respectively. We also know that a hyperk\"{a}hler manifold is the generalized version of
a K\"{a}hler manifold so that the notions of slant submanifolds, semi-slant submanifolds, pointwise slant submanifolds, pointwise semi-slant submanifolds
on a K\"{a}hler manifold can be induced on a hyperk\"{a}hler manifold such as h-slant submanifolds, h-semi-slant submanifolds, pointwise h-slant submanifolds, pointwise h-semi-slant submanifolds.

In this paper, we will study the notions of pointwise h-slant submanifolds and pointwise h-semi-slant submanifolds. But there are yet no papers
which are dealing with the notions of h-slant submanifolds and h-semi-slant submanifolds so that in the future works, we can consider such notions.
Given a pointwise almost h-slant submanifold $M$ of a hyperk\"{a}hler manifold $(\overline{M},I,J,K,g)$ such that $\{ I,J,K \}$ is a pointwise
almost h-slant basis, we have three almost h-slant functions $\theta_I, \theta_J, \theta_K$ on $M$.

Let $M$ be a pointwise almost h-semi-slant submanifold of an almost quaternionic Hermitian manifold $(\overline{M},E,g)$.
Given a point $p\in M$ with a neighborhood $V$, there exist an open set $U\subset \overline{M}$ with $U\cap M = V$ and a  quaternionic
Hermitian basis $\{ I,J,K \}$ of sections of $E$ on $U$ such that for each $R\in \{ I,J,K \}$,
there is a distribution $\mathcal{D}_1^R \subset TM$ on $V$ such that
$$
TM =\mathcal{D}_1^R\oplus \mathcal{D}_2^R, \
R(\mathcal{D}_1^R)=\mathcal{D}_1^R,
$$
and at each given point $q\in V$ the angle $\theta_R=\theta_R(X)$ between $RX$ and the space
$(\mathcal{D}_2^R)_q$ is constant for nonzero $X\in (\mathcal{D}_2^R)_q$,
where $\mathcal{D}_2^R$ is the orthogonal complement of $\mathcal{D}_1^R$ in $TM$.

Then for $X\in \Gamma(TM)$ and $R\in \{ I,J,K \}$, we write
\begin{equation}\label{eq: quat11}
X = P_RX+Q_RX,
\end{equation}
where $P_RX\in \Gamma(\mathcal{D}_1^R)$ and $Q_RX\in \Gamma(\mathcal{D}_2^R)$.

For $X\in \Gamma(TM)$ and $R\in \{ I,J,K \}$, we get
\begin{equation}\label{eq: quat12}
RX = \phi_R X+\omega_R X,
\end{equation}
where $\phi_R X\in \Gamma(TM)$ and $\omega_R X\in
\Gamma(TM^{\perp})$.

For $Z\in \Gamma(TM^{\perp})$ and $R\in \{ I,J,K \}$, we have
\begin{equation}\label{eq: quat13}
RZ = B_RZ+C_RZ,
\end{equation}
where $B_RZ\in \Gamma(TM)$ and $C_RZ\in \Gamma(TM^{\perp})$.

Then given $R\in \{ I,J,K \}$, we have
\begin{equation}\label{eq: quat15}
TM^{\perp} = \omega_R \mathcal{D}_2^R \oplus \mu_R,
\end{equation}
where $\mu_R$ is the orthogonal complement of $\omega_R
\mathcal{D}_2^R$ in $TM^{\perp}$ and is $R$-invariant.

Furthermore,

{\setlength\arraycolsep{2pt}
\begin{eqnarray}
& & \phi_R \mathcal{D}_1^R = \mathcal{D}_1^R, \ \omega_R
\mathcal{D}_1^R = 0, \ \phi_R \mathcal{D}_2^R \subset
\mathcal{D}_2^R, \ B_R(TM^{\perp}) = \mathcal{D}_2^R,   \label{eq: quat150}     \\
& & \phi_R^2+B_R\omega_R = -id, \ C_R^2+\omega_R B_R = -id,   \label{eq: quat1500}      \\
& & \omega_R \phi_R +C_R\omega_R = 0, \ B_RC_R+\phi_R B_R = 0.    \label{eq: quat15000}
\end{eqnarray}}

For $X,Y\in \Gamma(TM)$ and $R\in \{ I,J,K \}$, we define
\begin{equation}\label{eq: quat16}
(\nabla_X \phi_R)Y := \nabla_X (\phi_R Y) - \phi_R \nabla_X Y,
\end{equation}
\begin{equation}\label{eq: quat17}
(D_X \omega_R)Y := D_X (\omega_R Y) - \omega_R \nabla_X Y.
\end{equation}

Then the tensors $\phi_R$ and $\omega_R$ are said to be  {\em parallel} if $\nabla \phi_R = 0$ and $D \omega_R = 0$, respectively.

Throughout this paper, we will use the above notations.

We easily get

\begin{lemma}\label{eq: basic}
Let $M$ be a pointwise almost h-semi-slant submanifold of a hyperk\"{a}hler manifold $(\overline{M},I,J,K,g)$ such that $\{I,J,K\}$ is
a pointwise almost h-semi-slant basis on $\overline{M}$. Then we have

\begin{enumerate}
\item
\begin{align*}
  &(\nabla_X \phi_R) Y = A_{\omega_R Y} X + B_R h(X,Y),     \\
  &(D_X \omega_R) Y = -h(X,\phi_R Y) + C_R h(X,Y)
\end{align*}
for $X,Y\in \Gamma(TM)$ and $R\in \{ I,J,K \}$.

\item
\begin{align*}
  &-\phi_R A_Z X + B_R D_X Z = \nabla_X (B_R Z) - A_{C_R Z} X,    \\
  &-\omega_R A_Z X + C_R D_X Z = h(X,B_R Z) + D_X (C_R Z)
\end{align*}
for $X\in \Gamma(TM)$, $Z\in \Gamma(TM^{\perp})$,
and $R\in \{ I,J,K \}$.
\end{enumerate}
\end{lemma}

In a similar way to Proposition 3.6 of \cite{P3}, we obtain

\begin{proposition}\label{prop:slant}
Let $M$ be a pointwise almost h-semi-slant submanifold of an almost quaternionic Hermitian manifold $(\overline{M},E,g)$.
Then we get
\begin{equation}\label{eq: quat20}
\phi_R^2 X = -\cos^2 \theta_R \cdot X \quad \text{for} \ X\in
\Gamma(\mathcal{D}_2^R) \ \text{and} \ R\in \{ I,J,K \},
\end{equation}
where $\{ I,J,K \}$ is a pointwise almost h-semi-slant basis with the
almost h-semi-slant functions $\{ \theta_I,\theta_J,\theta_K \}$.
\end{proposition}

\begin{remark}
Since
\begin{equation}\label{eq: quat200}
g(\phi_R X, Y) = -g(X, \phi_R Y),
\end{equation}
by (\ref {eq: quat20}), we easily obtain
\begin{eqnarray}
& & g(\phi_R X, \phi_R Y) = \cos^2 \theta_R g(X, Y)   \label{eq: quat21}       \\
& & g(\omega_R X, \omega_R Y) = \sin^2 \theta_R g(X, Y)     \label{eq: quat22}
\end{eqnarray}
for $X,Y\in \Gamma(\mathcal{D}_2^R)$ and $R\in \{ I,J,K \}$.
\end{remark}

\begin{theorem}
Let $M$ be a pointwise h-semi-slant submanifold of a hyperk\"{a}hler manifold $(\overline{M},I,J,K,g)$ such that $\{I,J,K\}$ is
a pointwise h-semi-slant basis.
Then the following conditions are equivalent:

a) the complex distribution $\mathcal{D}_1$ is integrable.

b)
\begin{align*}
  &h(X, \phi_I Y) - h(Y, \phi_I X) + D_X (\omega_I Y) - D_Y (\omega_I X) = 0,     \\
  &Q_I\left(\nabla_X (\phi_I Y) - \nabla_Y (\phi_I X) + A_{\omega_I X} Y - A_{\omega_I Y} X\right) = 0
\end{align*}
for $X,Y\in \Gamma(\mathcal{D}_1).$

c)
\begin{align*}
  &h(X, \phi_J Y) - h(Y, \phi_J X) + D_X (\omega_J Y) - D_Y (\omega_J X) = 0,     \\
  &Q_J(\nabla_X (\phi_J Y) - \nabla_Y (\phi_J X) + A_{\omega_J X} Y - A_{\omega_J Y} X) = 0
\end{align*}
for $X,Y\in \Gamma(\mathcal{D}_1).$

d)
\begin{align*}
  &h(X, \phi_K Y) - h(Y, \phi_K X) + D_X (\omega_K Y) - D_Y (\omega_K X) = 0,     \\
  &Q_K(\nabla_X (\phi_K Y) - \nabla_Y (\phi_K X) + A_{\omega_K X} Y - A_{\omega_K Y} X) = 0
\end{align*}
for $X,Y\in \Gamma(\mathcal{D}_1).$
\end{theorem}

\begin{proof}
Given $X,Y\in \Gamma(\mathcal{D}_1)$ and $R\in \{ I,J,K \}$, by using (\ref{eq: quat12}), (\ref{eq: gauss}), and (\ref{eq: weing}), we have
\begin{align*}
R[X,Y]&= R(\overline{\nabla}_X Y - \overline{\nabla}_Y X)    \\
      &= \overline{\nabla}_X (RY) - \overline{\nabla}_Y (RX)    \\
      &= \overline{\nabla}_X (\phi_R Y + \omega_R Y) - \overline{\nabla}_Y (\phi_R X + \omega_R X)    \\
      &= \nabla_X (\phi_R Y) + h(X, \phi_R Y) - A_{\omega_R Y} X + D_X (\omega_R Y)     \\
      &- \nabla_Y (\phi_R X) - h(Y, \phi_R X) + A_{\omega_R X} Y - D_Y (\omega_R X).
\end{align*}
Since $\mathcal{D}_1$ is $R$-invariant, we obtain
$$
a) \Leftrightarrow b), \quad a) \Leftrightarrow c), \quad a)
\Leftrightarrow d).
$$
Therefore, the result follows.
\end{proof}

\begin{theorem}
Let $M$ be a pointwise h-semi-slant submanifold of a hyperk\"{a}hler manifold $(\overline{M},I,J,K,g)$ such that $\{I,J,K\}$ is
a pointwise h-semi-slant basis.
Then the following conditions are equivalent:

a) the slant distribution $\mathcal{D}_2$ is integrable.

b) $P_I(\nabla_X (\phi_I Y) - \nabla_Y (\phi_I X) + A_{\omega_I X} Y - A_{\omega_I Y} X) = 0$ for $X,Y\in \Gamma(\mathcal{D}_2)$.

c) $P_J(\nabla_X (\phi_J Y) - \nabla_Y (\phi_J X) + A_{\omega_J X} Y - A_{\omega_J Y} X) = 0$ for $X,Y\in \Gamma(\mathcal{D}_2)$.

d) $P_K(\nabla_X (\phi_K Y) - \nabla_Y (\phi_K X) + A_{\omega_K X} Y - A_{\omega_K Y} X) = 0$ for $X,Y\in \Gamma(\mathcal{D}_2)$.
\end{theorem}

\begin{proof}
Given $X,Y\in \Gamma(\mathcal{D}_2)$, $Z\in \Gamma(\mathcal{D}_1)$,
and $R\in \{ I,J,K \}$, by using (\ref{eq: quat12}), (\ref{eq: gauss}), and (\ref{eq: weing}), we get
\begin{align*}
g([X,Y], RZ)&= -g(R[X,Y], Z)        \\
      &= -g(\overline{\nabla}_X (RY) - \overline{\nabla}_Y (RX), Z)    \\
      &= -g(\overline{\nabla}_X (\phi_R Y + \omega_R Y) - \overline{\nabla}_Y (\phi_R X + \omega_R X), Z)   \\
      &= -g(\nabla_X (\phi_R Y) - A_{\omega_R Y} X - \nabla_Y (\phi_R X) + A_{\omega_R X} Y, Z).
\end{align*}
Since $[X,Y]\in \Gamma(TM)$ and $\mathcal{D}_1$ is $R$-invariant, we have
$$
a) \Leftrightarrow b), \quad a) \Leftrightarrow c), \quad a)
\Leftrightarrow d).
$$
Therefore, we obtain the result.
\end{proof}

Now we introduce some equivalent conditions for distributions $\mathcal{D}_i$ to be totally geodesic foliations for $i\in \{1,2\}$.

\begin{theorem}
Let $M$ be a pointwise h-semi-slant submanifold of a hyperk\"{a}hler manifold $(\overline{M},I,J,K,g)$ such that $\{I,J,K\}$ is
a pointwise h-semi-slant basis and all the h-semi-slant functions $\theta_I, \theta_J, \theta_K$ are nowhere zero.
Then the following conditions are equivalent:

a) the complex distribution $\mathcal{D}_1$ defines a totally geodesic foliation.

b) $g(h(X, Y), \omega_I \phi_I Z) = g(h(X, IY), \omega_I Z)$ for $X,Y\in \Gamma(\mathcal{D}_1)$ and $Z\in \Gamma(\mathcal{D}_2)$.

c)  $g(h(X, Y), \omega_J \phi_J Z) = g(h(X, JY), \omega_J Z)$ for $X,Y\in \Gamma(\mathcal{D}_1)$ and $Z\in \Gamma(\mathcal{D}_2)$.

d)  $g(h(X, Y), \omega_K \phi_K Z) = g(h(X, KY), \omega_K Z)$ for $X,Y\in \Gamma(\mathcal{D}_1)$ and $Z\in \Gamma(\mathcal{D}_2)$.
\end{theorem}

\begin{proof}
Given $X,Y\in \Gamma(\mathcal{D}_1)$, $Z\in \Gamma(\mathcal{D}_2)$, and $R\in \{ I,J,K \}$,
by using (\ref{eq: quat12}), (\ref{eq: quat20}), and  (\ref{eq: gauss}),
we have
\begin{align*}
g(\nabla_X Y, Z) &= g(\overline{\nabla}_X RY, RZ)    \\
                 &= g(\overline{\nabla}_X RY, \phi_R Z + \omega_R Z)    \\
                 &= -g(\overline{\nabla}_X Y, R\phi_R Z) + g(\overline{\nabla}_X RY, \omega_R Z)   \\
                 &= \cos^2 \theta_R\cdot g(\nabla_X Y, Z) - g(h(X, Y), \omega_R \phi_R Z) + g(h(X, RY), \omega_R Z)
\end{align*}
so that
$$
\sin^2 \theta_R\cdot g(\nabla_X Y, Z) = -g(h(X, Y), \omega_R \phi_R Z) + g(h(X, RY), \omega_R Z).
$$
Hence,
$$
a) \Leftrightarrow b), \quad a) \Leftrightarrow c), \quad a)
\Leftrightarrow d).
$$
Therefore, the result follows.
\end{proof}

\begin{theorem}
Let $M$ be a pointwise h-semi-slant submanifold of a hyperk\"{a}hler manifold $(\overline{M},I,J,K,g)$ such that $\{I,J,K\}$ is
a pointwise h-semi-slant basis and all the h-semi-slant functions $\theta_I, \theta_J, \theta_K$ are nowhere zero.
Then the following conditions are equivalent:

a) the slant distribution $\mathcal{D}_2$ defines a totally geodesic foliation.

b) $g(\omega_I \phi_I W, h(Z, X)) = g(\omega_I W, h(Z, IX))$ for $X\in \Gamma(\mathcal{D}_1)$ and $Z,W\in \Gamma(\mathcal{D}_2)$.

c) $g(\omega_J \phi_J W, h(Z, X)) = g(\omega_J W, h(Z, JX))$ for $X\in \Gamma(\mathcal{D}_1)$ and $Z,W\in \Gamma(\mathcal{D}_2)$.

d) $g(\omega_K \phi_K W, h(Z, X)) = g(\omega_K W, h(Z, KX))$ for $X\in \Gamma(\mathcal{D}_1)$ and $Z,W\in \Gamma(\mathcal{D}_2)$.
\end{theorem}

\begin{proof}
Given $X\in \Gamma(\mathcal{D}_1)$, $Z,W\in \Gamma(\mathcal{D}_2)$, and $R\in \{ I,J,K \}$,
by using (\ref{eq: quat12}), (\ref{eq: quat20}), and  (\ref{eq: gauss}),  we get
\begin{align*}
g(\nabla_Z W, X) &= -g(W, \nabla_Z X)      \\
                 &= -g(RW, \overline{\nabla}_Z RX)     \\
                 &= -g(\phi_R W + \omega_R W, \overline{\nabla}_Z RX)   \\
                 &= g(\phi^2_R W + \omega_R \phi_R W, \overline{\nabla}_Z X) - g(\omega_R W, \overline{\nabla}_Z RX)   \\
                 &= -\cos^2 \theta_R g(W, \nabla_Z X) + g(\omega_R \phi_R W, h(Z, X)) - g(\omega_R W, h(Z, RX))   \\
\end{align*}
so that
$$
\sin^2 \theta_R g(\nabla_Z W, X) = g(\omega_R \phi_R W, h(Z, X)) - g(\omega_R W, h(Z, RX)).
$$
Hence,
$$
a) \Leftrightarrow b), \quad a) \Leftrightarrow c), \quad a)
\Leftrightarrow d).
$$
Therefore, we obtain the result.
\end{proof}

\section{slant functions}\label{slant}

Like Proposition 2.1 of \cite{CG}, using (\ref{eq: quat21}) and (\ref{eq: quat150}), we easily obtain

\begin{proposition}
Let $M$ be a pointwise almost h-semi-slant submanifold of an almost quaternionic Hermitian manifold $(\overline{M},E,g)$.

Then
\begin{equation}\label{eq: quat33}
g(\phi_R X, \phi_R Y) = 0 \quad \text{whenever} \ g(X, Y) = 0
\end{equation}
for $X,Y\in \Gamma(TM)$ and $R\in \{ I,J,K \}$,
where $\{ I,J,K \}$ is a pointwise almost h-semi-slant basis.
\end{proposition}

\begin{proposition}
Let $M$ be a pointwise almost h-semi-slant submanifold of an almost quaternionic Hermitian manifold $(\overline{M},E,g)$.
Then given any $C^{\infty}$-function $f$ on $\overline{M}$, the submanifold $M$ is also a pointwise almost h-semi-slant submanifold of an almost quaternionic Hermitian manifold $(\overline{M},E,\widetilde{g})$ with $\widetilde{g} = e^{2f}g$.
\end{proposition}

\begin{proof}
Since $M$ is a pointwise almost h-semi-slant submanifold of an almost quaternionic Hermitian manifold $(\overline{M},E,g)$,
given a point $p\in M$ with a neighborhood $V$, we can take an open set $U\subset \overline{M}$ with $U\cap M = V$ and
a  quaternionic Hermitian basis $\{ I,J,K \}$ of
sections of $E$ on $U$ such that for each $R\in \{ I,J,K \}$,
there is a distribution $\mathcal{D}_1^R \subset TM$ on $V$
such that
$$
TM =\mathcal{D}_1^R\oplus \mathcal{D}_2^R, \
R(\mathcal{D}_1^R)=\mathcal{D}_1^R,
$$
and at each given point $q\in V$ the angle $\theta_R=\theta_R(X)$ between $RX$ and the space
$(\mathcal{D}_2^R)_q$ is constant for nonzero $X\in
(\mathcal{D}_2^R)_q$, where $\mathcal{D}_2^R$ is the
orthogonal complement of $\mathcal{D}_1^R$ in $TM$.

Given $X\in (\mathcal{D}_2^R)_q$ and $R\in \{ I,J,K \}$, we have
$$
\cos^2 \theta_R = \frac{g(\phi_R X, \phi_R X)}{g(X, X)}
$$
so that
$$
\cos^2 \theta_R = \frac{e^{2f}g(\phi_R X, \phi_R X)}{e^{2f}g(X, X)} = \frac{\widetilde{g}(\phi_R X, \phi_R X)}{\widetilde{g}(X, X)},
$$
which means the result.
\end{proof}

In a similar way to Proposition 4.1 of \cite{CG}, we obtain

\begin{proposition}
Let $M$ be a pointwise almost h-slant submanifold of a hyperk\"{a}hler manifold $(\overline{M},I,J,K,g)$
such that $\{I,J,K\}$ is a pointwise almost h-slant basis.
Then given $R\in \{ I,J,K \}$, the almost h-slant function $\theta_R$ is constant on $M$ if and only if
$$
A_{\omega_R X} \phi_R X = A_{\omega_R \phi_R X} X \quad \text{for} \ X\in \Gamma(TM).
$$
\end{proposition}

\begin{corollary}
Let $M$ be a pointwise almost h-slant submanifold of a hyperk\"{a}hler manifold $(\overline{M},I,J,K,g)$
such that $\{I,J,K\}$ is a pointwise almost h-slant basis.
Assume that $M$ is totally geodesic in $(\overline{M},g)$. Then the almost h-slant function $\theta_R$ is constant on $M$
for each $R\in \{ I,J,K \}$.
\end{corollary}

\section{Topological properties}\label{topol}

Let $M$ be a pointwise almost h-semi-slant submanifold of a hyperk\"{a}hler manifold $(\overline{M},I,J,K,g)$
such that $\{I,J,K\}$ is a pointwise almost h-semi-slant basis.
We call $M$ {\em proper} if $\theta_R(p) \in [0, \frac{\pi}{2})$ for $p\in M$ and $R\in \{ I,J,K \}$.

Let $M$ be a proper pointwise almost h-slant submanifold of a hyperk\"{a}hler manifold $(\overline{M},I,J,K,g)$
such that $\{I,J,K\}$ is a pointwise almost h-slant basis.

Define
\begin{equation}\label{eq: quat40}
\Omega_R (X, Y) := g(\phi_R X, Y)
\end{equation}
for $X,Y\in \Gamma(TM)$ and $R\in \{ I,J,K \}$.

By (\ref{eq: quat20}) and (\ref{eq: quat200}), $\Omega_R$ is a non-degenerate 2-form on $M$.

Using Lemma \ref{eq: basic}, we get

\begin{theorem}
Let $M$ be a proper pointwise almost h-slant submanifold of a hyperk\"{a}hler manifold $(\overline{M},I,J,K,g)$
such that $\{I,J,K\}$ is a pointwise almost h-slant basis.
Then the 2-form $\Omega_R$ is closed for each $R\in \{ I,J,K \}$.
\end{theorem}

\begin{proof}
Given $X,Y,Z\in \Gamma(TM)$ and $R\in \{ I,J,K \}$, we obtain
\begin{align*}
d\Omega_R (X,Y,Z) &= X\Omega_R (Y,Z) - Y\Omega_R (X,Z) + Z\Omega_R (X,Y)      \\
                  &- \Omega_R ([X,Y],Z) + \Omega_R ([X,Z],Y) - \Omega_R ([Y,Z],X)
\end{align*}
so that by (\ref{eq: quat150}) and (\ref{eq: quat16}),
\begin{align*}
d\Omega_R (X,Y,Z) &= g(\nabla_X \phi_R Y,Z) + g(\phi_R Y,\nabla_X Z) - g(\nabla_Y \phi_R X,Z)      \\
                  &- g(\phi_R X,\nabla_Y Z) + g(\nabla_Z \phi_R X,Y) + g(\phi_R X,\nabla_Z Y)  \\
                  &+ g([X,Y],\phi_R Z) - g([X,Z],\phi_R Y) + g([Y,Z],\phi_R X)      \\
                  &= g((\nabla_X \phi_R)Y,Z) - g((\nabla_Y \phi_R)X,Z) + g((\nabla_Z \phi_R)X,Y).      \\
\end{align*}
Using Lemma \ref{eq: basic} and (\ref{eq: shape}), we have
\begin{align*}
d\Omega_R (X,Y,Z) &= g(A_{\omega_R Y} X + B_R h(X,Y),Z) - g(A_{\omega_R X} Y + B_R h(Y,X),Z)     \\
                  &+ g(A_{\omega_R X} Z + B_R h(Z,X),Y)  \\
                  &= g(\omega_R Y,h(X,Z)) - g(h(X,Y),\omega_R Z)      \\
                  &- g(\omega_R X,h(Y,Z)) + g(h(Y,X),\omega_R Z)      \\
                  &+ g(\omega_R X,h(Z,Y)) - g(h(Z,X),\omega_R Y)      \\
                  &= 0.
\end{align*}
Therefore, we get the result.
\end{proof}

Denote by $[\Omega_R]$ the de Rham cohomology class of the 2-form $\Omega_R$ for $R\in \{ I,J,K \}$.
Then we obtain

\begin{theorem}\label{thm: coh}
Let $M$ be a $2n$-dimensional compact proper pointwise almost h-slant submanifold of a $4m$-dimensional hyperk\"{a}hler manifold $(\overline{M},I,J,K,g)$
such that $\{I,J,K\}$ is a pointwise almost h-slant basis.

Then
\begin{equation}\label{eq: quat41}
H^* (M, \mathbb{R}) \supseteq \widetilde{H},
\end{equation}
where $\widetilde{H}$ is the algebra spanned by $\{ [\Omega_I],[\Omega_J],[\Omega_K] \}$.
\end{theorem}

\begin{remark}

\begin{enumerate}

\item Let $M$ be a $2n$-dimensional proper pointwise almost h-slant submanifold of a $4m$-dimensional hyperk\"{a}hler manifold $(\overline{M},I,J,K,g)$
such that $\{I,J,K\}$ is a pointwise almost h-slant basis.

If $\theta_R (p) = 0$ for $p\in M$ and $R\in \{ I,J,K \}$, then $M$ is clearly a hyperk\"{a}hler manifold with $n$ even so that we may call $M$
the {\em generalized hyperk\"{a}hler manifold}.

\item Like Theorem \ref{thm: coh}, there are many results on the cohomology of compact hyperk\"{a}hler manifolds to be known
(\cite{H}, \cite{V}, \cite{F}, \cite{E}, \cite{G3}, \cite{P0}).

It is very interesting to compare Theorem \ref{thm: coh} with 1.6 of \cite{H}.

\end{enumerate}

\end{remark}

\begin{corollary}
Every $2n$-sphere $S^{2n}$, $n\geq 2$, cannot be immersed in a hyperk\"{a}hler manifold as a proper pointwise almost h-slant submanifold.
\end{corollary}

\begin{corollary}
Any $2n$-dimensional real projective space $\mathbb{R}\mathbb{P}^{2n}$, $n\geq 2$, cannot be immersed in a hyperk\"{a}hler manifold as a proper pointwise almost h-slant submanifold.
\end{corollary}

\begin{corollary}
Every $n$-dimensional complex projective space $\mathbb{C}\mathbb{P}^{n}$, $n\geq 2$, cannot be immersed in a hyperk\"{a}hler manifold as a proper pointwise almost h-slant submanifold.
\end{corollary}

\section{Warped product submanifolds}\label{warped}

In a similar way to Theorem 4.1 of \cite{S2}, we have

\begin{theorem}
Let $(\overline{M},I,J,K,g)$ be a hyperk\"{a}hler manifold. Then
given $R\in \{ I,J,K \}$, there do not exist any non-trivial
warped product submanifolds $M=B\times_f F$ of a K\"{a}hler
manifold $(\overline{M},R,g)$ such that $B$ is a proper pointwise
slant submanifold of $(\overline{M},R,g)$ and $F$ is a holomorphic submanifold of $(\overline{M},R,g)$.
\end{theorem}

\begin{proof}
Given $V\in \Gamma(TB)$, $X,Y\in \Gamma(TF)$, and $Z\in \Gamma(TM)$, we have
\begin{equation}\label{eq: prod01}
RZ = \phi Z + \omega Z,
\end{equation}
where $\phi Z\in \Gamma(TM)$ and $\omega Z\in \Gamma(TM^{\perp})$.

And we know that $\phi^2 V = -\cos^2 \theta V$ for a proper semi-slant
function $\theta$ on $M$ (i.e., $\theta : M \mapsto [0, \frac{\pi}{2})$).

Using (\ref {eq: warp2}), (\ref {eq: gauss}), (\ref {eq: weing}), (\ref {eq: shape}), we obtain
\begin{eqnarray*}
V(\ln f) g(X, Y) &=& -g(\overline{\nabla}_X (\phi^2 V + \omega\phi V), Y) - g(A_{\omega V} X, RY)   \\
                 &=& g(\overline{\nabla}_X (\cos^2 \theta V), Y) - g(\overline{\nabla}_X \omega\phi V, Y) - g(A_{\omega V} X, RY)   \\
                 &=& \cos^2 \theta g(\overline{\nabla}_X V, Y) + g(h(X, Y), \omega\phi V) - g(h(X, RY), \omega V)
\end{eqnarray*}
so that
\begin{equation}\label{eq: prod02}
\sin^2 \theta V(\ln f) g(X, Y) = g(h(X, Y), \omega\phi V) - g(h(X, RY), \omega V).
\end{equation}
Interchanging the role of $X$ and $Y$ at (\ref {eq: prod02}), we get
\begin{equation}\label{eq: prod03}
\sin^2 \theta V(\ln f) g(Y, X) = g(h(Y, X), \omega\phi V) - g(h(Y, RX), \omega V).
\end{equation}
Comparing (\ref {eq: prod02}) and (\ref {eq: prod03}), we have
\begin{equation}\label{eq: prod04}
g(h(X, RY), \omega V) = g(h(Y, RX), \omega V).
\end{equation}
But
\begin{eqnarray}
g(h(X, RY), \omega V) &=& g(\overline{\nabla}_X RY, RV - \phi V)  \label{eq: prod05}  \\
                 &=& g(\overline{\nabla}_X Y, V) + g(RY, \overline{\nabla}_X \phi V) \nonumber   \\
                 &=& -V(\ln f) g(X, Y) + \phi V(\ln f) g(X, RY).  \nonumber
\end{eqnarray}
From (\ref {eq: prod04}) and (\ref {eq: prod05}),
\begin{equation}\label{eq: prod06}
\phi V(\ln f) g(X, RY) = 0.
\end{equation}
Replacing $V$ by $\phi V$ and $X$ by $RX$ at (\ref {eq: prod06}), we obtain
\begin{equation}\label{eq: prod07}
\cos^2 \theta  V(\ln f) g(X, Y) = 0,
\end{equation}
which implies $V(\ln f) = 0$ so that $f$ is constant.
\end{proof}

\begin{corollary}
Let $(\overline{M},I,J,K,g)$ be a hyperk\"{a}hler manifold and
$M=B\times_f F$ a non-trivial warped product manifold of
Riemannian manifolds $(B,g_B)$ and $(F,g_F)$ with the warping
function $f$ on $B$. Then $M$ cannot be immersed in a
hyperk\"{a}hler manifold $(\overline{M},I,J,K,g)$ as a proper
pointwise h-semi-slant submanifold such that $TB = \mathcal{D}_2$,
$TF = \mathcal{D}_1$, and $\{ I,J,K \}$ is a pointwise
h-semi-slant basis.
\end{corollary}

\begin{example}
Let
$$
M := \{ (x_1,x_2,x_3,x_4,u,v) \mid 0< x_i < 1, 1\leq i \leq 4, 0 < u, v < \frac{\pi}{2} \}.
$$
Choose two points $P_1, P_2$ in the unit sphere $S^3$ such that
\begin{align*}
 &P_i = (a_{1i},a_{2i},a_{3i},a_{4i}), \quad  i = 1,2,     \\
 &\sum_{k=1}^4  a_{k1}a_{k2} = 0,  \\
 &-a_{11}a_{22} + a_{21}a_{12} - a_{31}a_{42} + a_{41}a_{32} \neq 0,      \\
 &-a_{11}a_{32} + a_{21}a_{42} + a_{31}a_{12} - a_{41}a_{22} \neq 0,     \\
 &-a_{11}a_{42} - a_{21}a_{32} + a_{31}a_{22} + a_{41}a_{12} \neq 0.  \\
\end{align*}
Define a map $i : M\subset \mathbb{R}^6 \mapsto \mathbb{R}^{20}$ by
\begin{align*}
 &i(x_1,x_2,x_3,x_4,u,v) = (y_1,y_2,\cdots,y_{20})    \\
 &= (x_1\cos u, x_2\cos u, x_3\cos u, x_4\cos u, x_1\cos v, x_2\cos v, x_3\cos v, x_4\cos v,   \\
 &x_1\sin u, x_2\sin u, x_3\sin u, x_4\sin u, x_1\sin v, x_2\sin v, x_3\sin v, x_4\sin v,       \\
 &a_{11}u + a_{12}v, a_{21}u + a_{22}v, a_{31}u + a_{32}v, a_{41}u + a_{42}v).     \\
\end{align*}
Consider a hyperk\"{a}hler structure $(I,J,K,\langle \ ,\ \rangle)$ on $\mathbb{R}^{20}$ (see Section 7).

Then the tangent bundle $TM$ (i.e., $i_* (TM)$) is generated by $X_1,X_2,X_3,X_4,Y_1,Y_2$,
where
\begin{align*}
 &X_1 = \cos u \tfrac{\partial}{\partial y_{1}} + \cos v \tfrac{\partial}{\partial y_{5}} +
 \sin u \tfrac{\partial}{\partial y_{9}} + \sin v \tfrac{\partial}{\partial y_{13}},    \\
 &X_2 = \cos u \tfrac{\partial}{\partial y_{2}} + \cos v \tfrac{\partial}{\partial y_{6}} +
 \sin u \tfrac{\partial}{\partial y_{10}} + \sin v \tfrac{\partial}{\partial y_{14}},    \\
 &X_3 = \cos u \tfrac{\partial}{\partial y_{3}} + \cos v \tfrac{\partial}{\partial y_{7}} +
 \sin u \tfrac{\partial}{\partial y_{11}} + \sin v \tfrac{\partial}{\partial y_{15}},    \\
 &X_4 = \cos u \tfrac{\partial}{\partial y_{4}} + \cos v \tfrac{\partial}{\partial y_{8}} +
 \sin u \tfrac{\partial}{\partial y_{12}} + \sin v \tfrac{\partial}{\partial y_{16}},    \\
 &Y_1 = -x_1\sin u \tfrac{\partial}{\partial y_{1}} -x_2 \sin u \tfrac{\partial}{\partial y_{2}} -
 x_3\sin u \tfrac{\partial}{\partial y_{3}} -x_4\sin u \tfrac{\partial}{\partial y_{4}}   \\
 &+ x_1\cos u \tfrac{\partial}{\partial y_{9}} + x_2\cos u \tfrac{\partial}{\partial y_{10}} +
 x_3\cos u \tfrac{\partial}{\partial y_{11}} + x_4\cos u \tfrac{\partial}{\partial y_{12}} \\
 &+ a_{11} \tfrac{\partial}{\partial y_{17}} + a_{21} \tfrac{\partial}{\partial y_{18}} +
 a_{31} \tfrac{\partial}{\partial y_{19}} + a_{41} \tfrac{\partial}{\partial y_{20}},   \\
 &Y_2 = -x_1\sin v \tfrac{\partial}{\partial y_{5}} -x_2 \sin v \tfrac{\partial}{\partial y_{6}} -
 x_3\sin v \tfrac{\partial}{\partial y_{7}} -x_4\sin v \tfrac{\partial}{\partial y_{8}}   \\
 &+ x_1\cos v \tfrac{\partial}{\partial y_{13}} + x_2\cos v \tfrac{\partial}{\partial y_{14}} +
 x_3\cos v \tfrac{\partial}{\partial y_{15}} + x_4\cos v \tfrac{\partial}{\partial y_{16}} \\
 &+ a_{12} \tfrac{\partial}{\partial y_{17}} + a_{22} \tfrac{\partial}{\partial y_{18}} +
 a_{32} \tfrac{\partial}{\partial y_{19}} + a_{42} \tfrac{\partial}{\partial y_{20}}.   \\
\end{align*}
It is easy to check that $M$ is a proper pointwise h-semi-slant submanifold of $(\mathbb{R}^{20},I,J,K,\langle \ ,\ \rangle)$ such that
$\mathcal{D}_1 = < X_1,X_2,X_3,X_4 >$, $\mathcal{D}_2 = < Y_1,Y_2 >$, and the h-semi-slant functions $\theta_I, \theta_J, \theta_K$ with
\begin{align*}
 &\cos \theta_I = \frac{|-a_{11}a_{22} + a_{21}a_{12} - a_{31}a_{42} + a_{41}a_{32}|}{1 + \sum_{k=1}^4 x_k^2},    \\
 &\cos \theta_J = \frac{|-a_{11}a_{32} + a_{21}a_{42} + a_{31}a_{12} - a_{41}a_{22}|}{1 + \sum_{k=1}^4 x_k^2},   \\
 &\cos \theta_K = \frac{|-a_{11}a_{42} - a_{21}a_{32} + a_{31}a_{22} + a_{41}a_{12}|}{1 + \sum_{k=1}^4 x_k^2}.
\end{align*}
Obviously, the distributions $\mathcal{D}_1$ and $\mathcal{D}_2$ are integrable so that we may denote by $B$ and $F$ the integral manifolds
of $\mathcal{D}_1$ and $\mathcal{D}_2$, respectively.

Then we can easily check that $M = (M, g)$ is a non-trivial warped product Riemannian submanifold of $\mathbb{R}^{20}$ such that
\begin{align*}
 &M = B \times_f F,    \\
 &g = 2(dx_1^2 + dx_2^2 + dx_3^2 + dx_4^2) + (1 + \sum_{k=1}^4 x_k^2)(du^2 + dv^2),   \\
 &\text{the warping function} \ f = \sqrt{1 + \sum_{k=1}^4 x_k^2}.
\end{align*}

Therefore, $M$ is a non-trivial warped product proper pointwise h-semi-slant submanifold of $(\mathbb{R}^{20},I,J,K,\langle \ ,\ \rangle)$.
\end{example}

Now we study non-trivial warped product proper pointwise h-semi-slant submanifolds of hyperk\"{a}hler manifolds.
Using these results, we will obtain Theorem \ref{thm: ineq33}.

\begin{lemma}
Let $M = B\times_f F$ be a non-trivial warped product proper pointwise h-semi-slant submanifold of a hyperk\"{a}hler manifold $(\overline{M},I,J,K,g)$
such that $TB = \mathcal{D}_1$, $TF = \mathcal{D}_2$, and $\{ I,J,K \}$ is a pointwise h-semi-slant basis.

Then we have
\begin{equation}\label{eq: quat50}
g(A_{\omega_R V} W, X) = g(A_{\omega_R W} V, X)
\end{equation}
for $V,W\in \Gamma(TF)$, $X\in \Gamma(TB)$, and $R\in \{ I,J,K \}$.
\end{lemma}

\begin{proof}
Given $V,W\in \Gamma(TF)$, $X\in \Gamma(TB)$, and $R\in \{ I,J,K \}$,
by using (\ref{eq: weing}), (\ref{eq: gauss}), (\ref{eq: quat150}), (\ref{eq: quat12}), (\ref{eq: warp2}), (\ref{eq: quat200}), and (\ref{eq: shape}),
we get
\begin{align*}
g(A_{\omega_R V} W, X) &= g(A_{\omega_R V} X, W)     \\
                  &= g(\nabla_X V, \phi_R W) + g(\overline{\nabla}_X V, \omega_R W) + g(\nabla_X \phi_R V, W)  \\
                  &= g(X(\ln f) V, \phi_R W) + g(h(X,V), \omega_R W) + g(X(\ln f) \phi_R V, W)      \\
                  &= g(A_{\omega_R W} V, X).
\end{align*}
\end{proof}

\begin{lemma}
Let $M = B\times_f F$ be a non-trivial warped product proper pointwise h-semi-slant submanifold of a hyperk\"{a}hler manifold $(\overline{M},I,J,K,g)$
such that $TB = \mathcal{D}_1$, $TF = \mathcal{D}_2$, and $\{ I,J,K \}$ is a pointwise h-semi-slant basis.

Then we obtain
\begin{equation}\label{eq: quat51}
g(A_{\omega_R \phi_R W} V, X) = -RX(\ln f) g(\phi_R W, V) - X(\ln f) \cos^2 \theta_R g(V, W)
\end{equation}
and
\begin{equation}\label{eq: quat52}
g(A_{\omega_R W} V, RX) = X(\ln f) g(W, V) + RX(\ln f) g(V, \phi_R W)
\end{equation}
for $V,W\in \Gamma(TF)$, $X\in \Gamma(TB)$, and $R\in \{ I,J,K \}$.
\end{lemma}

\begin{proof}
Given $V,W\in \Gamma(TF)$, $X\in \Gamma(TB)$, and $R\in \{ I,J,K \}$,
by using (\ref{eq: quat150}), (\ref{eq: quat50}), (\ref{eq: weing}), (\ref{eq: quat12}), (\ref{eq: gauss}), (\ref{eq: warp2}), and (\ref{eq: quat21}),
we have
\begin{align*}
g(A_{\omega_R \phi_R W} V, X) &= g(A_{\omega_R V} \phi_R W, X)     \\
                  &= -g(\overline{\nabla}_{\phi_R W} (RV - \phi_R V), X)  \\
                  &= g(\nabla_{\phi_R W} V, RX) + g(\nabla_{\phi_R W} \phi_R V, X)     \\
                  &= -g(V, RX(\ln f) \phi_R W) - g(\phi_R V, X(\ln f) \phi_R W)     \\
                  &= -RX(\ln f) g(\phi_R W, V) - X(\ln f) \cos^2 \theta_R g(V, W).
\end{align*}
Substituting $\phi_R W$ and $X$ by $W$ and $RX$, respectively, at (\ref{eq: quat51}),
we easily get (\ref{eq: quat52}).
\end{proof}

\begin{lemma}
Let $M = B\times_f F$ be a non-trivial warped product proper pointwise h-semi-slant submanifold of a hyperk\"{a}hler manifold $(\overline{M},I,J,K,g)$
such that $TB = \mathcal{D}_1$, $TF = \mathcal{D}_2$, and $\{ I,J,K \}$ is a pointwise h-semi-slant basis.

Then we obtain
\begin{equation}\label{eq: quat53}
g(h(X,Y), \omega_R V) = 0
\end{equation}
and
\begin{equation}\label{eq: quat54}
g(h(X,V), \omega_R W) = -RX(\ln f) g(V, W) + X(\ln f) g(V, \phi_R W)
\end{equation}
for $X,Y\in \Gamma(TB)$, $V,W\in \Gamma(TF)$, and $R\in \{ I,J,K \}$.
\end{lemma}

\begin{proof}
Given $X,Y\in \Gamma(TB)$, $V,W\in \Gamma(TF)$, and $R\in \{ I,J,K \}$,
by using (\ref{eq: gauss}), (\ref{eq: quat12}), (\ref{eq: quat150}), and (\ref{eq: warp2}),
we have
\begin{align*}
g(h(X,Y), \omega_R V) &= g(\overline{\nabla}_X Y, \omega_R V)  \\
                  &=  g(\overline{\nabla}_X Y, RV - \phi_R V)    \\
                  &= -g(\nabla_X RY, V) + g(Y, \nabla_X \phi_R V)     \\
                  &= g(RY, X(\ln f) V) + g(Y, X(\ln f) \phi_R V)       \\
                  &= 0.
\end{align*}
And replacing $X$ by $RX$ at (\ref{eq: quat52}), we easily get (\ref{eq: quat54}).
\end{proof}

Let $M = B\times_f F$ be a non-trivial warped product proper pointwise h-semi-slant submanifold of a hyperk\"{a}hler manifold $(\overline{M},I,J,K,g)$
such that $TB = \mathcal{D}_1$, $TF = \mathcal{D}_2$, $\dim B = 4n_1$, $\dim F = 2n_2$, $\dim \overline{M} = 4m$,
$\theta_I (p)\theta_J (p)\theta_K (p) \neq 0$ for any $p\in M$,  and $\{ I,J,K \}$ is a pointwise h-semi-slant basis.

Using (\ref{eq: quat21}) and (\ref{eq: quat22}), given $R\in \{ I,J,K \}$, we can choose a local orthonormal frame
$\{ e_1,e_{n_1+1},e_{2n_1+1},e_{3n_1+1},\cdots,e_{n_1},e_{2n_1},e_{3n_1},e_{4n_1},f_1^R,f_{n_2+1}^R,\cdots,f_{n_2}^R,f_{2n_2}^R,$  \\  $w_1^R,w_{n_2+1}^R,
\cdots,w_{n_2}^R,w_{2n_2}^R,v_1,Iv_1,Jv_1,Kv_1,\cdots,v_k,Iv_k,Jv_k,Kv_k \}$, $m = n_1 + n_2 + k$, on $(\overline{M},R,g)$ such that
$\{ e_1,e_{n_1+1},e_{2n_1+1},e_{3n_1+1},\cdots,e_{n_1},e_{2n_1},e_{3n_1},e_{4n_1} \}$ is a local orthonormal frame of $TB$,
$\{ f_1^R,f_{n_2+1}^R,\cdots,f_{n_2}^R,f_{2n_2}^R \}$ is a local orthonormal frame of $TF$,
$\{ w_1^R,w_{n_2+1}^R,\cdots,w_{n_2}^R,w_{2n_2}^R \}$ is a local orthonormal frame of $\omega_R (TF)$,
and $\{ v_1,Iv_1,Jv_1,Kv_1,\cdots,v_k,Iv_k,Jv_k,Kv_k \}$ is a local orthonormal frame of $\mu_R$,
where
\begin{align*}
                  & e_{n_1+i} := Ie_i, e_{2n_1+i} := Je_i, e_{3n_1+i} := Ke_i,     \\
                  &  f_{n_2+j}^R := \sec \theta_R \phi_R f_j^R,    \\
                  & w_j^R := \csc \theta_R \omega_R f_j^R,     \\
                  & w_{n_2+j}^R := \csc \theta_R \omega_R f_{n_2+j}^R = \csc \theta_R \sec \theta_R \omega_R \phi_R f_j^R
\end{align*}
for $1\leq i\leq n_1$ and $1\leq j\leq n_2$.

Using the above notations, we have

\begin{theorem}\label{thm: ineq33}
Let $M = B\times_f F$ be a non-trivial warped product proper pointwise h-semi-slant submanifold of a hyperk\"{a}hler manifold $(\overline{M},I,J,K,g)$
such that $TB = \mathcal{D}_1$, $TF = \mathcal{D}_2$, $\dim B = 4n_1$, $\dim F = 2n_2$, $\dim \overline{M} = 4m$,
$\theta_I (p)\theta_J (p)\theta_K (p) \neq 0$ for any $p\in M$,  and $\{ I,J,K \}$ is a pointwise h-semi-slant basis.

Assume that $m = n_1 + n_2$.

Then given $R\in \{ I,J,K \}$, we get
\begin{equation}\label{eq: quat55}
|| h ||^2 \geq 4n_2 (\csc^2 \theta_R + \cot^2 \theta_R) || \nabla (\ln f) ||^2
\end{equation}
with equality holding if and only if $g(h(V,W), Z) = 0$ for $V,W\in \Gamma(TF)$ and $Z\in \Gamma(TM^{\perp})$.
\end{theorem}

\begin{proof}
Given $R\in \{ I,J,K \}$, since $\mu_R = 0$, we obtain
\begin{align*}
|| h ||^2         &= \sum_{i,j=1}^{4n_1} g(h(e_i,e_j), h(e_i,e_j)) + \sum_{i,j=1}^{2n_2} g(h(f_i^R,f_j^R), h(f_i^R,f_j^R))      \\
                  &+  2\sum_{i=1}^{4n_1} \sum_{j=1}^{2n_2} g(h(e_i,f_j^R), h(e_i,f_j^R))    \\
                  &= \sum_{i,j=1}^{4n_1} \sum_{k=1}^{2n_2} g(h(e_i,e_j), w_k^R)^2 + \sum_{i,j=1}^{2n_2} \sum_{k=1}^{2n_2} g(h(f_i^R,f_j^R), w_k^R)^2      \\
                  &+  2\sum_{i=1}^{4n_1} \sum_{j=1}^{2n_2} \sum_{k=1}^{2n_2} g(h(e_i,f_j^R), w_k^R)^2.
\end{align*}
Using (\ref{eq: quat53}) and (\ref{eq: quat54}), we have
\begin{align*}
|| h ||^2         &= \sum_{i,j,k=1}^{2n_2} g(h(f_i^R,f_j^R), w_k^R)^2      \\
                  &+  2\csc^2 \theta_R \sum_{i=1}^{4n_1} \sum_{j,k=1}^{2n_2} \left(-Re_i(\ln f) g(f_j^R, f_k^R) + e_i(\ln f) g(f_j^R, \phi_R f_k^R)\right)^2   \\
                  &= \sum_{i,j,k=1}^{2n_2} g(h(f_i^R,f_j^R), w_k^R)^2      \\
                  &+  2\csc^2 \theta_R \sum_{i=1}^{4n_1} \sum_{j,k=1}^{2n_2} \left((Re_i(\ln f))^2 g(f_j^R, f_k^R)^2
                  + (e_i(\ln f))^2 g(f_j^R, \phi_R f_k^R)^2   \right. \\
                  &- \left. 2Re_i(\ln f)\cdot e_i(\ln f) g(f_j^R, f_k^R) g(f_j^R, \phi_R f_k^R) \right).
\end{align*}
But since $\nabla (\ln f)\in \Gamma(TB)$ and $R(TB) = TB$, we get
\begin{eqnarray}
\sum_{i=1}^{4n_1} Re_i(\ln f)\cdot e_i(\ln f)
                  & &= \sum_{i=1}^{4n_1} g(\nabla (\ln f), Re_i) g(\nabla (\ln f), e_i)   \label{eq: quat56}    \\
                  & &= -\sum_{i=1}^{4n_1} g(R\nabla (\ln f), e_i) g(\nabla (\ln f), e_i)  \nonumber \\
                  & &= -g(R\nabla (\ln f), \nabla (\ln f))   \nonumber   \\
                  & &=  0.  \nonumber
\end{eqnarray}
Moreover, we also obtain
\begin{eqnarray}
\sum_{i=1}^{4n_1} (Re_i(\ln f))^2
                  & &= \sum_{i=1}^{4n_1} g(Re_i, \nabla (\ln f))^2   \label{eq: quat57}    \\
                  & &= \sum_{i=1}^{4n_1} g(e_i, R\nabla (\ln f))^2 \nonumber  \\
                  & &= g(R\nabla (\ln f), R\nabla (\ln f))   \nonumber   \\
                  & &= g(\nabla (\ln f), \nabla (\ln f))   \nonumber   \\
                  & &= || \nabla (\ln f) ||^2.  \nonumber
\end{eqnarray}
Similarly,
\begin{equation} \label{eq: quat58}
\sum_{i=1}^{4n_1} (e_i(\ln f))^2 = || \nabla (\ln f) ||^2.
\end{equation}
Using (\ref{eq: quat56}), (\ref{eq: quat57}), (\ref{eq: quat58}), and (\ref{eq: quat21}), we have
\begin{align*}
|| h ||^2         &= \sum_{i,j,k=1}^{2n_2} g(h(f_i^R,f_j^R), w_k^R)^2        \\
                  &+ 2\csc^2 \theta_R (|| \nabla (\ln f) ||^2 2n_2 + || \nabla (\ln f) ||^2 \sum_{k=1}^{2n_2} g(\phi_R f_k^R, \phi_R f_k^R))   \\
                  &= \sum_{i,j,k=1}^{2n_2} g(h(f_i^R,f_j^R), w_k^R)^2        \\
                  &+ 2\csc^2 \theta_R (|| \nabla (\ln f) ||^2 2n_2 + || \nabla (\ln f) ||^2 \cos^2 \theta_R \cdot 2n_2)   \\
                  &= \sum_{i,j,k=1}^{2n_2} g(h(f_i^R,f_j^R), w_k^R)^2        \\
                  &+ 4n_2 (\csc^2 \theta_R + \cot^2 \theta_R) || \nabla (\ln f) ||^2
\end{align*}
so that
$$
|| h ||^2 \geq 4n_2 (\csc^2 \theta_R + \cot^2 \theta_R) || \nabla (\ln f) ||^2
$$
with equality holding if and only if $g(h(f_i^R,f_j^R), w_k^R) = 0$ for $1\leq i,j,k \leq 2n_2$.

Therefore, the result follows.
\end{proof}

Using (\ref{eq: quat53}) and Theorem \ref{thm: ineq33}, we get

\begin{corollary}
Let $M = B\times_f F$ be a non-trivial warped product proper pointwise h-semi-slant submanifold of a hyperk\"{a}hler manifold $(\overline{M},I,J,K,g)$
such that $TB = \mathcal{D}_1$, $TF = \mathcal{D}_2$, $\dim B = 4n_1$, $\dim F = 2n_2$, $\dim \overline{M} = 4m$,
$\theta_I (p)\theta_J (p)\theta_K (p) \neq 0$ for any $p\in M$,  and $\{ I,J,K \}$ is a pointwise h-semi-slant basis.

Assume that $m = n_1 + n_2$.

If $|| h ||^2 = 4n_2 (\csc^2 \theta_R + \cot^2 \theta_R) || \nabla (\ln f) ||^2$ for some $R\in \{ I,J,K \}$,
then $M$ is a minimal submanifold of $\overline{M}$.
\end{corollary}

\section{Examples}\label{exam}

Note that given an Euclidean space $\mathbb{R}^{4m}$ with
coordinates $(y_1,y_2,\cdots,y_{4m})$, we choose
complex structures $I, J, K$ on $\mathbb{R}^{4m}$ as follows:
\begin{align*}
  &I(\tfrac{\partial}{\partial y_{4k+1}})=\tfrac{\partial}{\partial y_{4k+2}},
  I(\tfrac{\partial}{\partial y_{4k+2}})=-\tfrac{\partial}{\partial y_{4k+1}},
  I(\tfrac{\partial}{\partial y_{4k+3}})=\tfrac{\partial}{\partial y_{4k+4}},
  I(\tfrac{\partial}{\partial y_{4k+4}})=-\tfrac{\partial}{\partial y_{4k+3}},     \\
  &J(\tfrac{\partial}{\partial y_{4k+1}})=\tfrac{\partial}{\partial y_{4k+3}},
  J(\tfrac{\partial}{\partial y_{4k+2}})=-\tfrac{\partial}{\partial y_{4k+4}},
  J(\tfrac{\partial}{\partial y_{4k+3}})=-\tfrac{\partial}{\partial y_{4k+1}},
  J(\tfrac{\partial}{\partial y_{4k+4}})=\tfrac{\partial}{\partial y_{4k+2}},    \\
  &K(\tfrac{\partial}{\partial y_{4k+1}})=\tfrac{\partial}{\partial y_{4k+4}},
  K(\tfrac{\partial}{\partial y_{4k+2}})=\tfrac{\partial}{\partial y_{4k+3}},
  K(\tfrac{\partial}{\partial y_{4k+3}})=-\tfrac{\partial}{\partial y_{4k+2}},
  K(\tfrac{\partial}{\partial y_{4k+4}})=-\tfrac{\partial}{\partial y_{4k+1}}
\end{align*}
for $k\in \{ 0,1,\cdots,m-1 \}$.
Then we easily check that $(I,J,K,\langle \ ,\ \rangle)$ is a hyperk\"{a}hler structure on $\mathbb{R}^{4m}$,
where $\langle \ ,\ \rangle$ denotes the Euclidean metric on $\mathbb{R}^{4m}$.
Throughout this section, we will use these notations.

\begin{example}
Let $(\overline{M},E,g)$ be an almost quaternionic Hermitian manifold.
Then the tangent bundle $T\overline{M}$ of the manifold $\overline{M}$ has the natural almost quaternionic Hermitian structure
such that $\overline{M}$ is a pointwise h-slant submanifold of $T\overline{M}$ with $\theta = 0$ \cite{IMV}.
\end{example}

\begin{example}
Let $(\overline{M},E,g)$ be an almost quaternionic Hermitian manifold.
Let $M$ be a submanifold of $\overline{M}$ such that $(M,E|_M,g|_M)$ is an almost quaternionic Hermitian manifold,
where $E|_M$ and $g|_M$ denote the restrictions of $E$ and $g$ to $M$, respectively.
Then $M$ is a  pointwise h-slant submanifold of $\overline{M}$ with $\theta = 0$.
\end{example}

\begin{example}
Let $M$ be a submanifold of a hyperk\"{a}hler manifold $(\overline{M},I,J,K,g)$. Assume that the submanifold $M$ is complex
with respect to the complex structure $I$ (i.e., $I(TM) = TM$) and totally real with respect to the complex structure
$J$ (i.e., $J(TM) \subset TM^{\perp}$).
Then it is easy to check that $M$ is also
totally real with respect to the complex structure $K$. Given any $C^{\infty}$-function $f : \overline{M} \mapsto [0, \frac{\pi}{2}]$, we define
\begin{align*}
  &\overline{I} := \cos f \cdot I - \sin f \cdot J,     \\
  &\overline{J} := \sin f \cdot I + \cos f \cdot J,    \\
  &\overline{K} := K.
\end{align*}
It is also easy to show that $\{ \overline{I},\overline{J},\overline{K} \}$ is a quaternionic Hermitian basis on $(\overline{M},g)$.
Then $M$ is a pointwise almost h-slant submanifold of an almost quaternionic Hermitian manifold $(\overline{M},\overline{I},\overline{J},\overline{K},g)$
such that $\{ \overline{I},\overline{J},\overline{K} \}$ is a pointwise almost h-slant basis with the almost h-slant functions
$$
\theta_{\overline{I}} = f, \quad \theta_{\overline{J}} = \frac{\pi}{2} - f, \quad \theta_{\overline{K}} = \frac{\pi}{2}.
$$
\end{example}

\begin{example}
Let $(\overline{M}_1,E,g_1)$ be an almost quaternionic Hermitian manifold.
Let $M_1$ be a submanifold of $\overline{M}_1$ such that $(M_1,E|_{M_1},g_1|_{M_1})$ is an almost quaternionic Hermitian manifold,
where $E|_{M_1}$ and $g_1|_{M_1}$ denote the restrictions of $E$ and $g_1$ to $M_1$, respectively.
Let $M_2$ be a submanifold of a hyperk\"{a}hler manifold $(\overline{M}_2,I,J,K,g_2)$ such that the submanifold $M_2$ is complex
with respect to the complex structure $I$ and totally real with respect to the complex structure $J$.
Given any $C^{\infty}$-function $f : \overline{M}_2 \mapsto [0, \frac{\pi}{2}]$, we define
\begin{align*}
  &\overline{I} := \cos f \cdot I - \sin f \cdot J,     \\
  &\overline{J} := \sin f \cdot I + \cos f \cdot J,    \\
  &\overline{K} := K.
\end{align*}
Consider $\overline{M}_1 = (\overline{M}_1,E,g_1)$ and $\overline{M}_2 = (\overline{M}_2,\overline{I},\overline{J},\overline{K},g_2)$.
Let $h : \overline{M}_1 \mapsto \mathbb{R}$ be a positive $C^{\infty}$-function.
Denote by $M_1 \times_h M_2$ and $\overline{M}_1 \times_h \overline{M}_2$ the warped product manifolds of $M_1$, $M_2$
and $\overline{M}_1$, $\overline{M}_2$, respectively.
Then $M_1 \times_h M_2$ is a pointwise h-semi-slant submanifold of $\overline{M}_1 \times_h \overline{M}_2$ such that
$\mathcal{D}_1 = TM_1$, $\mathcal{D}_2 = TM_2$, $\theta_{\overline{I}} = f$, $\theta_{\overline{J}} = \frac{\pi}{2} - f$,
$\theta_{\overline{K}} = \frac{\pi}{2}$.
\end{example}

\begin{example}
Define a map $i : \mathbb{R}^{4} \mapsto \mathbb{R}^{8}$ by
$$
i(x_1,x_2,x_3,x_{4}) = (y_1,y_2,\cdots, y_{8}) = (0,0,x_3,x_1,0,x_4, x_2,0).
$$
Then $\mathbb{R}^{4}$ is a pointwise almost h-semi-slant submanifold of $\mathbb{R}^{8}$ such that
\begin{align*}
&\mathcal{D}_1^I = < \frac{\partial}{\partial y_3},\frac{\partial}{\partial y_4} >, \quad
  \mathcal{D}_2^I = < \frac{\partial}{\partial y_6},\frac{\partial}{\partial y_7} >,   \\
&\mathcal{D}_2^J = < \frac{\partial}{\partial y_3},\frac{\partial}{\partial y_4},\frac{\partial}{\partial y_6},\frac{\partial}{\partial y_7} >,   \\
&\mathcal{D}_1^K = < \frac{\partial}{\partial y_6},\frac{\partial}{\partial y_7} >, \quad
  \mathcal{D}_2^K = < \frac{\partial}{\partial y_3},\frac{\partial}{\partial y_4} >   \\
\end{align*}
with the almost h-semi-slant functions $\theta_I = \theta_J = \theta_K = \frac{\pi}{2}$.
\end{example}

\begin{example}
Let $f : \mathbb{R}^{8} \mapsto [0, \frac{\pi}{2}]$ be a $C^{\infty}$-function.
Consider a quaternionic Hermitian basis $\{ \overline{I},\overline{J},\overline{K} \}$ on $\mathbb{R}^{8}$ such that
\begin{align*}
  &\overline{I} := \cos f \cdot I - \sin f \cdot J,     \\
  &\overline{J} := \sin f \cdot I + \cos f \cdot J,    \\
  &\overline{K} := K.
\end{align*}
Define a map $i : \mathbb{R}^{6} \mapsto \mathbb{R}^{8}$ by
$$
i(x_1,x_2,\cdots,x_{6}) = (y_1,y_2,\cdots, y_{8}) = (0,0,x_4,x_1,x_5,x_2,x_6,x_3).
$$
Then $\mathbb{R}^{6}$ is a pointwise h-semi-slant submanifold of $(\mathbb{R}^{8},\overline{I},\overline{J},\overline{K},\langle \ ,\ \rangle)$ such that
\begin{align*}
&\mathcal{D}_1 = < \frac{\partial}{\partial y_5},\frac{\partial}{\partial y_6},\frac{\partial}{\partial y_7},\frac{\partial}{\partial y_8} >,       \\
&\mathcal{D}_2 = < \frac{\partial}{\partial y_3},\frac{\partial}{\partial y_4} >   \\
\end{align*}
with the h-semi-slant functions $\theta_{\overline{I}} = f$, $\theta_{\overline{J}} = \frac{\pi}{2} - f$, $\theta_{\overline{K}} = \frac{\pi}{2}$.
\end{example}

\section*{Acknowledgments}

The author is grateful to the referees for their valuable comments and suggestions.

\end{document}